\documentclass{ws-ijnt_ARXIV}

\begin{document}

\markboth{Yiwang Chen, Nicholas Dunn, Campbell Hewett, Shashwat Silas}
{On the Inversion Polynomial for Dedekind Sums}

%
\catchline{}{}{}{}{}
%

\title{ON THE INVERSION POLYNOMIAL FOR DEDEKIND SUMS}

\author{Yiwang Chen}

\address{University of Illinois at Urbana-Champaign\\
\email{ychen137@illinois.edu} }

\author{Nicholas Dunn}

\address{North Carolina State University at Raleigh\\
\email{njdunn2@ncsu.edu} }

\author{Campbell Hewett}

\address{Brown University\\
\email{campbell\_hewett@brown.edu} }

\author{Shashwat Silas}

\address{Brown University\\
\email{shashwat\_silas@brown.edu} }

\maketitle

\begin{abstract}
We introduce the inversion polynomial for Dedekind sums $f_b(x)=\sum x^{\inv(a,b)}$ to study the number of $s(a,b)$ which have the same value for given $b$. We prove several properties of this polynomial and present some conjectures. We also introduce connections between Kloosterman sums and the inversion polynomial evaluated at particular roots of unity. Finally, we improve on previously known bounds for the second highest value of the Dedekind sum and provide a conjecture for a possible generalization. Lastly, we include a new restriction on equal Dedekind sums based on the reciprocity formula.
\end{abstract}

\section{Introduction}	

\subsection{The Dedekind Sum}

Let $a$ and $b$ be relatively prime integers. The \textit{Dedekind sum} $s(a,b)$ is given by
\begin{equation}
s(a,b)=\sum_{i=1}^{b-1}\frac{i}{b}\llrr{\frac{ai}{b}},
\end{equation}
where 
\begin{equation}
\llrr{x}=\begin{cases}0 & x \in \Z \\ x-\lfloor x \rfloor-\frac{1}{2} & x \not\in \Z\end{cases}
\end{equation}
is the sawtooth function. Our motivating problem is to classify all $a_1$ and $a_2$ such that $s(a_1,b)=s(a_2,b)$ for a fixed $b$. Partial results are given by the following theorem in [1].

\textbf{Theorem.}
$s(a_1,b)-s(a_2,b) \in \Z$ if and only if $b|(a_1-a_2)(a_1 a_2-1)$.

The problem is solved in the case where $b$ is a prime power in [2].

\subsection{The Inversion Number}

Given a permutation $\sigma$ of the numbers $\{1,2,\dots,b\}$, define the \textit{inversion number} of $\sigma$ as 
\begin{equation}
\inv(\sigma)=\#\{(i,j)~|~1 \le i<j \le b,\sigma(i)>\sigma(j)\}.
\end{equation}
If $a$ is relatively prime to $b$, then the map $\sigma_a:\{1,2,\dots,b\} \to \{1,2,\dots,b\}$ given by $x \mapsto ax \pmod{b}$ is a permutation of $\{1,2,\dots,b\}$. In this case define $\inv(a,b)=\inv(\sigma_a)$. The following is a fact due to Zolotarev.

\textbf{Theorem.}
For $(a,b)=1$, 
\begin{equation}\label{invtos}
\inv(a,b)=-3bs(a,b)+\frac{(b-1)(b-2)}{4}.
\end{equation}

From this, one sees that for a given $b$, $s(a_1,b)=s(a_2,b)$ if and only if $\inv(a_1,b)=\inv(a_2,b)$. 

\subsection{The Inversion Polynomial}
An advantage of the inversion number is that it is always a nonnegative integer. Therefore, for a positive integer $b$, define the \textit{inversion polynomial}

\begin{equation}
f_b(x)=\sum_{\substack{1 \le a \le b\\(a,b)=1}} x^{\inv(a,b)}.
\end{equation}

We focus most of our attention on exploring this polynomial because it suggests how often Dedekind sums take on certain values. Indeed, if $c x^d$ is a term in $f_b(x)$, then exactly $c$ different values of $a$ have $\inv(a,b)=d$. There are existing bounds on the number of such $a$ that can take on the same Dedekind sum; as mentioned in [2], if $b$ is square-free, it can be shown that for a given $d$, the number of $a$ such that $s(a,b)=d$ cannot exceed $2^r$, where $r$ is the number of prime factors of $b$. A complete understanding of the inversion polynomial will give us the number of equal Dedekind sums. In our paper, we have been able to characterize many polynomial factors of $f_b$. For certain $x$, we also show that $f_b(x)$ can be written in terms of Kloosterman sums. We also present novel bounds on the smallest values of $\inv(a,b)$ for a given $b$.

\subsection{Elementary Properties of the Inversion Polynomial}

A list of facts about the inversion polynomial is as follows.

(i) The degree of $f_b$ is $\frac{(b-1)(b-2)}{2}$. The largest value $\inv(a,b)$ can take is 
\begin{equation}
\inv(-1,b)=\frac{(b-1)(b-2)}{2}
\end{equation}
because for every pair $1 \le i < j \le b-1$, $\sigma_{-1}(i)>\sigma_{-1}(j)$.

(ii) The constant term and the leading coefficient are both 1. If $\inv(a,b)=0$, then we must have $a=1$, and if $\inv(a,b)=\frac{(b-1)(b-2)}{2}$, then we must have $a=-1$. Translating this result into Dedekind sums using (\ref{invtos}), we obtain a simple proof of the known fact
\begin{equation}
-\frac{(b-1)(b-2)}{12b} \le \inv(a,b) \le \frac{(b-1)(b-2)}{12b}.
\end{equation}

(iii) The coefficients are symmetric. This follows from
\begin{equation}
\inv(a,b)+\inv(-a,b)=\frac{(b-1)(b-2)}{2},
\end{equation}
which is true because for $1 \le i<j \le b-1$, exactly one of $\sigma_a(i)>\sigma_a(j)$ and $\sigma_{-a}(i)>\sigma_{-a}(j)$ holds.

(iv) For $b$ not divisible by three, there is a polynomial $g$ such that $f_b(x)=g(x^3)$. In other words, if $b$ is not divisible by three, $\inv(a,b)$ is divisible by three. To see this, rewrite (\ref{invtos}) as
\begin{equation}
2\inv(a,b)=3\left(-2bs(a,b)+\frac{(b-1)(b-2)}{6}\right).
\end{equation}
When three does not divide $b$, the quantity in the parentheses is an integer, as shown in [4]. Three does not divide two, so three divides $\inv(a,b)$.

\section{Overview of Results and Conjectures}

\subsection{Factors of the Inversion Polynomial}

From decomposing the inversion polynomial into its irreducible factors for many $b$, it appears that $f_b(x)$ is the product of cyclotomic polynomials as well as exactly one non-cyclotomic irreducible factor. The following is a conjecture (several parts of which are proved) which characterizes most of the cyclotomic factors. Throughout this paper, we use $\zeta_m$ to denote a primitive $m$th root of unity.

\begin{conjecture}
For $b=ck$, $\zeta_{2k}$ and $\zeta_{6k}$ are roots of $f_b(x)$ precisely under the following conditions:

(i) If $c=2 \pmod{4}$, then $\zeta_{2k}$ and $\zeta_{6k}$ are roots of $f_b(x)$ precisely when $k=4 \pmod{8}$.

(ii) If $c=0 \pmod{4}$, then $\zeta_{2k}$ and $\zeta_{6k}$ are roots of $f_b(x)$ precisely when $k \neq 2 \pmod{4}$ and $8 \nmid k$.

(iii) If $c=3^m n$, $m \ge 1$, and $(n,6)=1$, then only $\zeta_{2k}$ is a root of $f_b(x)$ precisely when $3n \nmid k$ and $c$ is not a square.

(iv) If $(c,6)=1$, then $\zeta_{2k}$ and $\zeta_{6k}$ are roots of $f_b(x)$ precisely when $c \nmid k$ and $c$ is not a square.
\end{conjecture}

For some $b$, there do exist other roots $\zeta_m$, where $m \nmid 2b$ and $m \nmid 6b$. Table 1 below shows the first several roots $\zeta_m$ not accounted for in Conjecture 2.1.

\begin{table}[ht]
\tbl{All roots $\zeta_m$ not accounted for in Conjecture 2.1 up to $b=424$.}
{\begin{tabular}{c|c}
$b$ & Unexplained roots \\\hline
8 & $\zeta_{18}$\\
18 & $\zeta_{16}$\\
22 & $\zeta_{20}, \zeta_{60}$\\
26 & $\zeta_{20}, \zeta_{60}$\\
29 & $\zeta_{18}$\\
40 & $\zeta_{18}, \zeta_{90}$\\
45 & $\zeta_{8}, \zeta_{40}$\\
46 & $\zeta_{36}$\\
56 & $\zeta_{18}$\\
57 & $\zeta_{54}$\\
70 & $\zeta_{18}, \zeta_{90}$\\
74 & $\zeta_{36}$\\
80 & $\zeta_{14}, \zeta_{42}$\\
83 & $\zeta_{18}$\\
114 & $\zeta_{108}$\\
117 & $\zeta_{8}$
\end{tabular}
\begin{tabular}{c|c}
$b$ & Unexplained roots \\\hline
136 & $\zeta_{18}, \zeta_{306}$\\
138 & $\zeta_{108}$\\
148 & $\zeta_{18}, \zeta_{36}$\\
173 & $\zeta_{18}$\\
186 & $\zeta_{20}$\\
198 & $\zeta_{20}$\\
200 & $\zeta_{18}$\\
204 & $\zeta_{54}$\\
296 & $\zeta_{18}$\\
317 & $\zeta_{18}$\\
325 & $\zeta_{18}, \zeta_{90}$\\
332 & $\zeta_{72}$\\
345 & $\zeta_{54}$\\
362 & $\zeta_{36}$\\
398 & $\zeta_{18}$\\
424 & $\zeta_{18}$
\end{tabular}}
\end{table}

Below we list several particular cases of Conjecture 2.1 whose proofs are known.

\begin{proposition}
We can completely characterize the integers $b$ for which $x+1$ divides $f_b(x)$:
\begin{equation}
f_b(-1)=\begin{cases}0 & \text{ if $b$ is an odd non-square or $4|b$}\\ \varphi(b) & \text{ if $b$ is an odd square or $b=2 \pmod{4}$}\end{cases}.
\end{equation}
When $b$ is an odd square or $b=2 \pmod{4}$, $\inv(a,b)$ takes only even values.
\end{proposition}

The next result is not a special case of Conjecture 2.1 but is a classification of more particular roots.

\begin{proposition}
For $b$ any nonsquare integer $1 \pmod{4}$, $-1$ is a double root of $f_b(x)$. If in addition three does not divide $b$, then $\zeta_6$ is a double root of $f_b(x)$. That is, $(x^3+1)^2$ divides $f_b(x)$.
\end{proposition}

When $b$ is odd, the situation becomes more manageable. Our next result is explained by Conjecture 2.1 and includes one direction of Proposition 2.2 as a special case.

\begin{proposition}
If $b=ck$, where $b$ is odd, $c$ is not square, and $(c,k)=1$, then $\zeta_{2k}$ is a root $f_b(x)$. If $b$ is not divisible by three, then $\zeta_{6k}$ is also a root of $f_b(x)$.
\end{proposition}

For certain $b$, $f_b$ takes values in terms of Kloosterman sums when evaluated on roots of unity. Here we denote the Kloosterman sum
\begin{equation}
K(a,b,m)=\sum_{\substack{1 \le x \le m \\(x,m)=1}}e^{\frac{2\pi i}{m}(ax+bx^{-1})},
\end{equation}
where $x^{-1}$ is the inverse of $x$ modulo $m$.

\begin{proposition}
Write $b=ck$. If $c=0 \pmod{4}$, we have
\begin{equation}
f_b(e^{2\pi i/2k})=\frac{1}{2}e^{\frac{2\pi i}{4k}}K\left(\frac{b}{4k},\frac{b}{4k},2b\right).
\end{equation}
So, (ii) in Conjecture 2.1 now says that if $4k|b$, then $K(b/4k,b/4k,2b)=0$ precisely when $k \neq 2 \pmod{4}$ and $8 \nmid k$.

If $k$ is even and $c=2 \pmod{4}$, we have
\begin{equation}
f_b(e^{2\pi i/2k})=\frac{1}{4}ie^{\frac{2\pi i}{4k}}K\left(\frac{b}{2k},\frac{b}{2k}(1-b),4b\right).
\end{equation}
Case (i) in Conjecture 2.1 says $k$ must be even, so it now becomes the statement that if $k$ is even and $2k|b$ but $4k \nmid b$, then $K(b/2k,b(1-b)/2k,4b)=0$ precisely when $k=4 \pmod{8}$.
\end{proposition}

\subsection{Smallest Values of Inversion Number}

For a given $b$, we know that the smallest value of inversion number occurs at $\inv(1,b)$ and the largest at $\inv(-1,b)$. Here we prove a bound on the second smallest and second largest values, which improves on bounds given in Section 6 of [5].

\begin{proposition}
If $2 \le a \le b-2$, then 
\begin{equation}
\frac{b^2-1}{8} \le \inv(a,b) \le \frac{3b^2-12b+9}{8}.
\end{equation}
If $b$ is odd, then
\begin{equation}
\inv(2,b)=\frac{b^2-1}{8}, \hspace{1cm}\text{ and }\hspace{1cm} \inv(-2,b)=\frac{3b^2-12b+9}{8}.
\end{equation}

Translating this result into Dedekind sums using (\ref{invtos}), we get
\begin{equation}
-\frac{(b-1)(b-5)}{24b} \le s(a,b) \le \frac{(b-1)(b-5)}{24b},
\end{equation}
where
\begin{equation}
s(2,b)=\frac{(b-1)(b-5)}{24b}, \hspace{1cm}\text{ and }\hspace{1cm} s(-2,b)=-\frac{(b-1)(b-5)}{24b}.
\end{equation}
\end{proposition}

The above proposition is part of a larger conjecture about the six smallest values of inversion number.

\begin{conjecture}
Let $I_n(b)$ be the $n$th smallest value of $\inv(a,b)$ for fixed $b$ and $1 \le n \le 6$. Then
\begin{equation}
I_n(b) \ge \frac{(n-1)(b+1)(b+1-n)}{4n},
\end{equation}
with equality occurring with 
\begin{equation}
\inv(n,b)=\frac{(n-1)(b+1)(b+1-n)}{4n}
\end{equation}
when $b=-1 \pmod{n}$.
\end{conjecture}

We suspect that this conjecture can be proved in a manner similar to the proof of Proposition 2.6 given below.

The following result is a novel restriction of $a_1$ and $a_2$ such that $s(a_1,b)=s(a_2,b)$.

\begin{proposition}
Suppose we have $a_1$ and $a_2$ such that $b=r \pmod{a_1}$, $b=r \pmod{a_2}$, $a_1=a_2=s\pmod{r}$, and $\inv(a_1,b)=\inv(a_2,b)$. Then $a_1=a_2$.
\end{proposition}

\section{Proofs of Propositions}

Many of the proofs rely on the following well-known theorem, known as the reciprocity formula for Dedekind sums.

\begin{theorem}
For $(a,b)=1$, 
\begin{equation}\label{recip1}
s(a,b) + s(b,a) = \frac{1}{12}\left(\frac{a}{b}+\frac{1}{ab}+\frac{b}{a}\right)-\frac{1}{4}.
\end{equation}
In terms of inversion number, (\ref{invtos}) and (\ref{recip1}) become
\begin{equation}\label{recip2}
a\inv(a,b)+b\inv(b,a)=\frac{(a-1)(b-1)(a+b-1)}{4}.
\end{equation}
\end{theorem}

\begin{proof}[Proof of Proposition 2.2]
First take the case where $b$ is odd. Then we use Zolotarev's theorem
\begin{equation}
(-1)^{\inv(a,b)}=\left(\frac{a}{b}\right),
\end{equation}
where $\left(\frac{a}{b}\right)$ is the Jacobi symbol, to compute
\begin{equation}
f_b(-1)=\sum_{(a,b)=1}(-1)^{\inv(a,b)}=\sum_{(a,b)=1}\left(\frac{a}{b}\right)=\begin{cases}0 & \text{ if $b$ is not a square}\\\varphi(b) & \text{ if $b$ is a square}\end{cases}.
\end{equation}

Now take the case where $b$ is even. If $4|b$, then from 
\begin{equation}
\inv(a,b)+\inv(-a,b)=\frac{(b-1)(b-2)}{2}=1 \pmod{2},
\end{equation}
we know one of $\inv(a,b)$ and $\inv(-a,b)$ is even and the other is odd. It follows that
\begin{equation}
f_b(-1)=\sum_{(a,b)=1}(-1)^{\inv(a,b)}=0.
\end{equation}
If $b=2 \pmod{4}$, then from reducing (\ref{recip2}) modulo 2,
\begin{align}
\inv(a,b)=\frac{1}{4}a^{-1}(a-1)(b-1)(a+b-1)=\left(\frac{a-1}{2}\right)\left(\frac{a-1}{2}+\frac{b}{2}\right) \nonumber\\
=\left(\frac{a-1}{2}\right)\left(\frac{a-1}{2}+1\right)=0 \pmod{2}.
\end{align}
It follows that $f_b(-1)=\varphi(b)$.
\end{proof}

\begin{proof}[Proof of Proposition 2.3]
First we show that $-1$ is a root of $f_b'(x)$. Then we show that when $b$ is not divisible by three, $e^{\pi i/3}$ is a root of both $f_b(x)$ and $f_b'(x)$. First compute
\begin{equation}
\inv(-a,b)\left(\frac{-a}{b}\right)=\left(\frac{(b-1)(b-2)}{2}-\inv(a,b)\right)\left(\frac{a}{b}\right),
\end{equation}
because of (6) and $\left(\frac{-1}{b}\right)=1$. In other words,
\begin{equation}
\inv(-a,b)\left(\frac{-a}{b}\right)+\inv(a,b)\left(\frac{a}{b}\right)=\frac{(b-1)(b-2)}{2}\left(\frac{a}{b}\right).
\end{equation}
Write
\begin{equation}
x f_b'(x)=\sum_{(a,b)=1}\inv(a,b) x^{\inv(a,b)}=\frac{1}{2}\left(\sum_{(a,b)=1}\inv(a,b) x^{\inv(a,b)}+\sum_{(a,b)=1}\inv(-a,b) x^{\inv(-a,b)}\right)
\end{equation}
so that
\begin{align}
-2f_b'(-1)&=\sum_{(a,b)=1}\inv(a,b) \left(\frac{a}{b}\right)+\sum_{(a,b)=1}\inv(-a,b) \left(\frac{-a}{b}\right) \nonumber\\
&=\frac{(b-1)(b-2)}{2}\sum_{(a,b)=1}\left(\frac{a}{b}\right)=0.
\end{align}
Now assume $b$ is not divisible by three. Then, by (iv) in Section 1.4,
\begin{align}
f_b(e^{\pi i/3})&=\sum_{(a,b)=1}(e^{\pi i/3})^{\inv(a,b)}=\sum_{(a,b)=1}(e^{\pi i})^{\inv(a,b)/3} \nonumber\\
&=\sum_{(a,b)=1}(-1)^{\inv(a,b)/3}=\sum_{(a,b)=1}((-1)^3)^{\inv(a,b)/3}=0.
\end{align}
Using the same reasoning, we see that $f_b'(e^{\pi i/3})=0$.
\end{proof}

\begin{proof}[Proof of Proposition 2.4]
Reduce (\ref{recip2}) modulo $k$ to get
\begin{equation}
4a\inv(a,b)=-(a-1)^2 \pmod{k}.
\end{equation}
Then, because $b$ is odd, we have from Zolotarev's theorem
\begin{equation}
\inv(a,b)=\frac{1}{2}\left(\left(\frac{a}{b}\right)-1\right) \pmod{2}
\end{equation}
and the Chinese remainder theorem,
\begin{equation}
\inv(a,b)=-(4a)^{-1}(a-1)^2(k+1) + \frac{1}{2}\left(\left(\frac{a}{b}\right)-1\right)k \pmod{2k}.
\end{equation}
Now we show that $e^{2\pi i/(2k)}$ is a root of $f_b(x)$ and $e^{2\pi i/(6k)}$ is a root as well if $b$ is not divisible by 3. Compute
\begin{align}
f_b(e^{2\pi i/(2k)})&=\sum_{(a,b)=1} e^{2\pi i \inv(a,b) \frac{1}{2k}} \nonumber\\
&=\sum_{(a,b)=1} e^{\pi i \frac{1}{k}\left(-(4a)^{-1}(a-1)^2(k+1) + \frac{1}{2}\left(\left(\frac{a}{b}\right)-1\right)k\right)} \nonumber\\
&=\sum_{(a,b)=1} e^{\pi i \frac{1}{k}\left(-(4a)^{-1}(a-1)^2(k+1)\right)}(-1)^{\frac{1}{2}\left(\left(\frac{a}{b}\right)-1\right)} \nonumber\\
&=\sum_{(a,b)=1} e^{\pi i \frac{1}{k}\left(-(4a)^{-1}(a-1)^2(k+1)\right)}\left(\frac{a}{b}\right).
\end{align}
Now, write $a=qk+r$ for $q$ and $r$ such that $(r,k)=1$ and $(qk+r,c)=1$. This is a valid parameterization of the values of $a$ because $(c,k)=1$. The above becomes
\begin{align}
&=\sum_{(r,k)=1} \sum_{(qk+r,c)=1} e^{\pi i\frac{1}{k}\left(-(4(qk+r))^{-1}(qk+r-1)^2(k+1)\right)}\left(\frac{qk+r}{ck}\right) \nonumber\\
&=\sum_{(r,k)=1} \sum_{(qk+r,c)=1} e^{\pi i\frac{1}{k}\left(-(4r)^{-1}(qk+r-1)^2(k+1)\right)}\left(\frac{qk+r}{c}\right)\left(\frac{r}{k}\right) \nonumber\\
&=\sum_{(r,k)=1} e^{\pi i \frac{1}{k}\left(-(4r)^{-1}(r-1)^2(k+1)\right)}\left(\frac{r}{k}\right)\sum_{(qk+r,c)=1}e^{\pi i\frac{1}{k}\left(-(4r)^{-1}\big(q^2k^2+2qk(r-1)\big)(k+1)\right)}\left(\frac{qk+r}{c}\right) \nonumber\\
&=\sum_{(r,k)=1} e^{\pi i \frac{1}{k}\left(-(4r)^{-1}(r-1)^2(k+1)\right)}\left(\frac{r}{k}\right)\sum_{(qk+r,c)=1}e^{\pi i\left(-(4r)^{-1}\big(q^2k+2q(r-1)\big)(k+1)\right)}\left(\frac{qk+r}{c}\right) \nonumber\\
&=\sum_{(r,k)=1} e^{\pi i \frac{1}{k}\left(-(4r)^{-1}(r-1)^2(k+1)\right)}\left(\frac{r}{k}\right)\sum_{(qk+r,c)=1}e^{\pi i\left(-(4r)^{-1}q^2k(k+1)\right)}\left(\frac{qk+r}{c}\right).
\end{align}
Then, $k(k+1)$ is always even, so this is
\begin{align}
&=\sum_{(r,k)=1} e^{\pi i \frac{1}{k}\left(-(4r)^{-1}(r-1)^2(k+1)\right)}\left(\frac{r}{k}\right)\sum_{(qk+r,c)=1}\left(\frac{qk+r}{c}\right)=0
\end{align}
because $c$ is not a square. If $b$ is not divisible by three, then $e^{2\pi i/(6k)}$ is a root for the same reason as in the proof of Proposition 2.3.
\end{proof}

\begin{proof}[Proof of Proposition 2.5]
If $c=0\pmod{4}$, then $4k|b$. Then we know
\begin{equation}
\inv(a,b)=-a^{-1}\left(\frac{a-1}{2}\right)^2 \pmod{2k}
\end{equation}
by reducing (\ref{recip2}) modulo $2k$. Here we denote $a^{-1}$ as the inverse of $a$ modulo $4b$. Then,
\begin{align}
f_b(e^{2\pi i/2k}) &= \sum_{(a,b)=1} e^{\frac{2\pi i}{2k}\left(-a^{-1}\left(\frac{a-1}{2}\right)^2\right)} \nonumber\\
&=\sum_{(a,b)=1} e^{\frac{2\pi i}{8k}\left(-a^{-1}(a^2-2a+1)\right)} \nonumber\\
&=e^{\frac{2\pi i}{4k}}\sum_{(a,b)=1}e^{\frac{2\pi i}{8k}(a+a^{-1})} \nonumber\\
&=e^{\frac{\pi i}{4k}}\sum_{(a,b)=1}e^{\frac{2\pi i}{8b}c(a+a^{-1})}. \label{prop25ii}
\end{align}
Because $c=0 \pmod{4}$, we have $(a+b)^{-1}=a^{-1}+b \pmod{2b}$. To see this, 
\begin{equation}
(a+b)(a^{-1}+b)=1+(a+a^{-1})b+b^2=1 \pmod{2b}
\end{equation}
because $a$ and $a^{-1}$ are both odd and $b$ is even.
\begin{align}
K\left(\frac{c}{4},\frac{c}{4},2b\right)&=\sum_{(a,2b)=1}e^{\frac{2\pi i}{8b}c(a+a^{-1})} \nonumber\\
&=\sum_{(a,b)=1}e^{\frac{2\pi i}{8b}c(a+a^{-1})}+\sum_{(a,b)=1}e^{\frac{2\pi i}{8b}c((a+b)+(a+b)^{-1})} \nonumber\\
&=\sum_{(a,b)=1}e^{\frac{2\pi i}{8b}c(a+a^{-1})}+e^{\frac{2\pi i}{4}c}\sum_{(a,b)=1}e^{\frac{2\pi i}{8b}c(a+a^{-1})} \nonumber\\
&=2\sum_{(a,b)=1}e^{\frac{2\pi i}{8b}c(a+a^{-1})}.
\end{align}
Hence, (\ref{prop25ii}) is
\begin{equation}
f_b(e^{2\pi i/2k})=\frac{1}{2}e^{\frac{2\pi i}{4k}}K\left(\frac{b}{4k},\frac{b}{4k},2b\right).
\end{equation}
Now assume $k$ is even. If $c=2 \pmod{4}$, we need to use
\begin{equation}
\inv(a,b)=-a^{-1}\left(\frac{a-1}{2}\right)^2-a^{-1}\left(\frac{a-1}{2}\right)\frac{b}{2} \pmod{2k},
\end{equation}
which is again (\ref{recip2}) reduced modulo $2k$. So then the analogous form of (\ref{prop25ii}) is
\begin{align}
f_b(e^{2\pi i/2k}) &= \sum_{(a,b)=1} e^{\frac{2\pi i}{2k}\left(-a^{-1}\left(\frac{a-1}{2}\right)^2-a^{-1}\left(\frac{a-1}{2}\right)\frac{b}{2}\right)} \nonumber\\
&=\sum_{(a,b)=1}e^{-\frac{2\pi i}{8k}a^{-1}(a^2-2a+1+ab-b)} \nonumber\\
&=\sum_{(a,b)=1}e^{-\frac{2\pi i}{8k}(a-2+a^{-1}+b-ba^{-1})} \nonumber\\
&=e^{\frac{2\pi i}{8k}(2-b)}\sum_{(a,b)=1}e^{\frac{2\pi i}{8k}(a+(1-b)a^{-1})} \nonumber\\
&=i e^{\frac{2\pi i}{4k}}\sum_{(a,b)=1}e^{\frac{2\pi i}{8b}c(a+(1-b)a^{-1})}. \label{prop25i}
\end{align}
Following the steps we did before, we see that $(a+jb)^{-1}=a^{-1}-jb \pmod{4b}$. To see this, 
\begin{equation}
(a+jb)(a^{-1}-jb)=1-(a-a^{-1})jb-j^2b^2=1-a^{-1}(a^2-1)jb-j^2b^2=1 \pmod{4b}
\end{equation}
because $a^2-1$ and $b$ are both divisible by four. Then,
\begin{align}
K\left(\frac{c}{2},\frac{c}{2}(1-b),4b\right)&=\sum_{(a,4b)=1}e^{\frac{2\pi i}{8b}c(a+(1-b)a^{-1})} \nonumber\\
&=\sum_{j=0}^3\sum_{(a,b)=1}e^{\frac{2\pi i}{8b}c((a+jb)+(1-b)(a+jb)^{-1})} \nonumber\\
&=\sum_{j=0}^3\sum_{(a,b)=1}e^{\frac{2\pi i}{8b}c((a+jb)+(1-b)(a^{-1}-jb))} \nonumber\\
&=\sum_{j=0}^3\sum_{(a,b)=1}e^{\frac{2\pi i}{8b}c(a+(1-b)a^{-1}+jb^2)} \nonumber\\
&=4\sum_{(a,b)=1}e^{\frac{2\pi i}{8b}c(a+(1-b)a^{-1})}.
\end{align}
So (\ref{prop25i}) becomes
\begin{equation}
f_b(e^{2\pi i/2k})=\frac{1}{4}ie^{\frac{2\pi i}{4k}}K\left(\frac{b}{2k},\frac{b}{2k}(1-b),4b\right).
\end{equation}
\end{proof}

\begin{proof}[Proof of Proposition 2.6]
Because the polynomial is symmetric, it is enough to show the first inequality. We proceed by induction on $b$. This can be easily verified for small values of $b$. Now, suppose the statement is true up to $b-1$. First take the case $b \neq \pm 1 \pmod{a}$. Then from 
\begin{equation}
\inv(a,b)=\frac{(a-1)(b-1)(a+b-1)}{4a}-\frac{b}{a}\inv(b,a)
\end{equation}
we get
\begin{align}
\inv(a,b) &\ge \frac{(a-1)(b-1)(a+b-1)}{4a}-\frac{b}{a}\left(\frac{3a^2-12a+9}{8}\right) \nonumber\\
&=-\frac{1}{8}(b+2)a+\frac{1}{4}(b^2+3b+2)-\frac{1}{8}(2b^2+5b+2)\frac{1}{a}=:F(a).
\end{align}
We may compute the minimum value $F(a)$ takes on the interval $3 \le a \le b-2$ (note that if $b \neq \pm 1 \pmod{a}$, then $a \neq 2$). Its derivative is
\begin{equation}
\frac{d}{da}F(a)=-\frac{1}{8}(b+2)+\frac{1}{8}(2b^2+5b+2)\frac{1}{a^2}.
\end{equation}
This has a zero at $a=\sqrt{2b+1}$, which is inside the interval, but it is a maximum because the second derivative is negative:
\begin{equation}
\frac{d^2}{da^2}F(a)=-\frac{1}{4}(2b^2+5b+2)\frac{1}{a^3}.
\end{equation}
So, we check the end points
\begin{equation}
F(3)=\frac{b^2+b-2}{6}=\frac{b^2+4b-5}{24}+\frac{b^2-1}{8}
\end{equation}
and
\begin{equation}
F(b-2)=\frac{b^3+2b^2-9b-18}{8b-16}=\frac{b^2-2b+5}{2b-4}+\frac{b^2-1}{8}.
\end{equation}
The fraction $(b^2+4b-5)/24$ is positive for $b>1$, so $F(3)>(b^2-1)/8$. The fraction $(b^2-2b+5)/(2b-4)$ is always positive as well. To see this, note that the largest root of $b^2-2b+5$ is $1+\sqrt{6} \approx 3.449$, so for $b \ge 4$, the numerator and denominator are both positive. This tells us that for $b \ge 4$, $F(b-2) > (b^2-1)/8$.

Now take the case where $b = 1\pmod{a}$. Then $\inv(b,a)=\inv(1,a)=0$, so
\begin{equation}
\inv(a,b)=\frac{(a-1)(b-1)(a+b-1)}{4a}=\frac{1}{4}(b-1)a+\frac{1}{4}(b^2-3b+2)-\frac{1}{4}(b^2-2b+1)\frac{1}{a}=:G(a).
\end{equation}
Note that $G(a)$ is increasing for $b \ge 2$ because its derivative is positive:
\begin{equation}
\frac{d}{da}G(a)=\frac{1}{4}(b-1)+\frac{1}{4}(b-1)^2\frac{1}{a^2}.
\end{equation}
So,
\begin{equation}
\inv(a,b) \ge G(2)=\frac{b^2-1}{8}.
\end{equation}
For the last case, $b=-1 \pmod{a}$, we have $\inv(b,a)=\inv(-1,a)=(a-1)(a-2)/2$, so
\begin{align}
\inv(a,b)&=\frac{(a-1)(b-1)(a+b-1)}{4a}-\frac{b}{a}\frac{(a-1)(a-2)}{2} \nonumber\\
&=-\frac{1}{4}(b+1)a+\frac{1}{4}(b^2+3b+2)-\frac{1}{4}(b^2+2b+1)\frac{1}{a}=:H(a).
\end{align}
As in the case of $F(a)$, we can compute the minimum value it takes in the interval $2 \le a \le b-2$. Compute the derivative
\begin{equation}
\frac{d}{da}H(a)=-\frac{1}{4}(b+1)+\frac{1}{4}(b+1)^2\frac{1}{a^2}.
\end{equation}
This has a zero at $a=\sqrt{b+1}$, which is inside the interval for $b \ge 2$, but it is a maximum because the second derivative is negative:
\begin{equation}
\frac{d^2}{da^2}H(a)=-\frac{1}{2}(b+1)^2\frac{1}{a^3}.
\end{equation}
So, we check the end points
\begin{equation}
H(2)=\frac{b^2-1}{8}.
\end{equation}
The other end point actually does get smaller than $(b^2-1)/8$, but if $b=-1 \pmod{a}$, then $a \le (b+1)/2$. Thus, we only need to look at the end point
\begin{equation}
H\left(\frac{b+1}{2}\right)=\frac{b^2-1}{8}.
\end{equation}
Therefore, in all cases, $\inv(a,b)$ gets no lower than $(b^2-1)/8$.
\end{proof}

\begin{proof}[Proof of Proposition 2.7]
From repeated use of (\ref{recip2}), one can show inductively that if we have the following sequence
\begin{equation}
r_{-1}=b,r_0=a,r_1,r_2,r_2,\dots,r_n,r_{n+1}=1
\end{equation}
where $r_{j+2}$ is $r_j$ reduced modulo $r_{j+1}$ for all $-1 \le j \le n-1$, then 
\begin{equation}
\inv(a,b)=\frac{a-1}{4a}b^2+\left(\frac{(a-1)(a-2)}{4a}+\frac{1}{4}\sum_{j=1}^n \frac{(-1)^j}{r_j r_{j-1}}(r_j-1)(r_{j-1}-1)(r_j+r_{j-1}-1)\right)b-\frac{(a-1)^2}{4a}.
\end{equation}
From this, if $\inv(a_1,b)=\inv(a_2,b)$,
\begin{align}
0&=\left(\frac{a_1-1}{4a_1}-\frac{a_2-1}{4a_2}\right)b^2+\left(\frac{(a_1-1)(a_1-2)}{4a_1}-\frac{(a_2-1)(a_2-2)}{4a_2}-\frac{(a_1-1)(r-1)(a_1+r-1)}{4a_1r} \right. \nonumber\\
&\hspace{2cm}\left. +\frac{(a_2-1)(r-1)(a_2+r-1)}{4a_2r}\right)b-\left(\frac{(a_1-1)^2}{4a_1}-\frac{(a_2-1)^2}{4a_2}\right) \nonumber\\
&=-\frac{1}{4a_1}-\frac{a_1}{4}+\frac{1}{4a_2}+\frac{a_2}{4}-\frac{b^2}{4a_1}+\frac{b^2}{4a_2}+\frac{b}{4a_1r}+\frac{a_1 b}{4r}-\frac{b}{4a_2r}-\frac{a_2b}{4r}+\frac{br}{4a_1}-\frac{br}{4a_2} \nonumber\\
&=\frac{1}{a_1}\left(-\frac{1}{4}+\frac{b}{4r}\right)(a_1-a_2)\left(a_1-\frac{1}{a_2}+\frac{br}{a_2}\right).
\end{align}
This means that $a_1=a_2$ or $a_1=(1-br)/a_2$. The second one is impossible if $1 \le a_1,a_2 \le b-1$, so $a_1=a_2$.
\end{proof}

\section{Directions for Further Research}

As mentioned above, understanding all of the factors of $f_b$ will tell us the number of equal Dedekind sums. In this effort, it would be of great interest to discover a more general version of Conjecture 2.1 which covers the cyclotomic roots not classified in our paper. It would also be necessary to understand the large non-cyclotomic irreducible factor of the inversion polynomial. From analyzing the roots of $f_b$ for many $b$, some patterns are apparent. For example, all roots which are not on the unit circle belong to the large non-cyclotomic irreducible factor; however, this factor does contain roots on the unit circle as well. Fig. 1 below are plots of the roots of $f_b$ for $b=11$, 14, and 21.

\begin{figure}[h!]
\centerline{\psfig{file=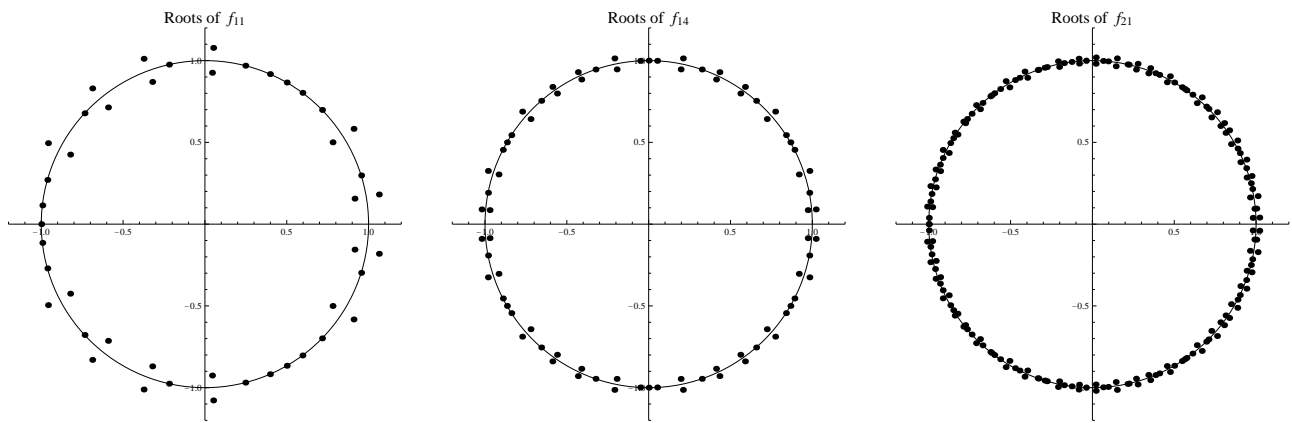,width=7in}} 
\vspace*{8pt}
\caption{The roots of $f_{11}$, $f_{14}$, and $f_{21}$ respectively, plotted on the complex plane.\label{fig1}}
\end{figure}

One can show, using Proposition 2.6 and Rouche's theorem that if $x$ is a root of $f_b$, then 
\begin{equation}
e^{-\frac{8\log \varphi(b)}{b^2-1}} < |x| < e^{\frac{8\log \varphi(b)}{b^2-1}}.
\end{equation}
By analyzing the roots of the $f_b$, it may be possible to arrive at bounds for the coefficients.

\section{Acknowledgements}

We would like the thank Professor Sinai Robins for his invaluable advice and guidance through this project. He introduced us to the Dedekind sums problem through his previous work, [3]. We also appreciate advice and gracious help of Le Quang Nhat and Emmanuel Tsukerman. This project was funded by ICERM, the Brown UTRA program and the Brown University Department of Mathematics. We are very grateful for their trust and support.


\begin{thebibliography}{0}

\bibitem{1} K. Girstmair, A criterion for the equality of Dedekind sums, Int. J. Number Theory, 10, 565 (2014). DOI: 10.1142/S179304211350108X.

\bibitem{2} K. Girstmair, On Dedekind sums with equal values, arxiv: 1404.4428

\bibitem{3} S. Jabuka, S. Robins, X. Wang, When are two Dedekind sums equal?, arXiv:1103.0370. International Journal of Number Theory Vol. 7, No. 8 (2011) 2197-2202.

\bibitem{4} H. Rademacher, Dedekind sums, The Carus Mathematical Monographs, Number Sixteen.

\bibitem{5} E. Tsukerman, Fourier-Dedekind sums and an extension of Rademacher reciprocity, arXiv: 1307.5542. To appear in The Ramanujan Journal.

\end{thebibliography}
\end{document}